\documentclass[11pt,a4paper,twoside]{article}
\usepackage[T1]{fontenc}
\usepackage[english]{babel}
\usepackage[utf8]{inputenc}
\usepackage[pdftex]{graphicx}
\usepackage{amsmath,amsthm,amssymb,amsfonts,amsbsy}
\usepackage{lmodern}

\usepackage{color, colortbl} 
\usepackage{float} 
\usepackage{bookmark} 
\usepackage{microtype} 
\usepackage[margin=2.5cm]{geometry} 
\usepackage{booktabs} 
\usepackage{cite}
\usepackage{authblk} 

\theoremstyle{definition}
\newtheorem{definition}{Definition}[section]

\theoremstyle{plain}

\newtheorem{theorem}[definition]{Theorem}
\newtheorem{proposition}[definition]{Proposition}
\newtheorem{corollary}[definition]{Corollary}
\newtheorem{conjecture}[definition]{Conjecture}

\DeclareMathOperator{\GG}{GG}
\DeclareMathOperator{\NGG}{NGG}
\newcommand{\ABC}{{\rm ABC}}
\usepackage{mathtools}
\DeclarePairedDelimiter\ceil{\lceil}{\rceil}
\DeclarePairedDelimiter\floor{\lfloor}{\rfloor}

\title{Remarks on the Graovac-Ghorbani Index of Bipartite Graphs}

\author{
    \large \bf Darko Dimitrov$^a$, Barbara Ikica$^b$, Riste {\v S}krekovski$^c$
    }

\affil{  \normalsize
    $^a${ Hochschule f\"ur Technik und Wirtschaft Berlin, Germany
    }
    \\E-mail: {\tt darko.dimitrov11@gmail.com}
}

\affil{
    $^b${ Faculty of Mathematics and Physics, University of Ljubljana \& \\
    Institute of Mathematics, Physics and Mechanics, Ljubljana, Slovenia}
    \\E-mail: {\tt barbara.ikica@fmf.uni-lj.si}
}

\affil{
    $^c${ Faculty of Information Studies, Novo mesto \& \\
    Faculty of Mathematics and Physics, University of Ljubljana \&  \\ Faculty of Mathematics, Natural Sciences and Information Technologies, \\University of Primorska, Koper, Slovenia}
    \\E-mail: {\tt skrekovski@gmail.com}
}

\date{\today}

\begin{document}

\maketitle

\begin{abstract}
    The atom-bond connectivity ($\ABC$) index is a well-known degree-based molecular structure descriptor with a variety of chemical applications. In $2010$ Graovac and Ghorbani introduced a distance-based analog of this index, the Graovac-Ghorbani ($\GG$) index, which yielded promising results when compared to analogous descriptors. In this paper, we investigate the structure of graphs that maximize and minimize the $\GG$ index. Specifically, we show that amongst all bipartite graphs, the minimum $\GG$ index is attained by a complete bipartite graph, while the maximum $\GG$ index is attained by a path or a cycle-like graph; the structure of the resulting graph depends on the number of vertices. Through the course of the research, we also derive an asymptotic estimate of the $\GG$ index of paths. In order to obtain our results, we introduce a normalized version of the $\GG$ index and call it the normalized Graovac-Ghorbani ($\NGG$) index. Finally, we discuss some related open questions as a potential extension of our work.
\end{abstract}


\medskip

\section[Introduction]{Introduction}

Let $G$ be a simple undirected graph of order $n=|V(G)|$ and size $m=|E(G)|$. The degree of a vertex $v \in V(G)$, denoted by $d(v)$, is the number of edges incident to $v$. Then the {\it atom-bond connectivity (ABC) index} of $G$ is defined as
\begin{equation*}
    \ABC(G)=\sum_{uv\in E(G)}\sqrt{\frac{d(u) +d(v)-2}{d(u)d(v)}}.
\end{equation*}
This degree-based molecular structure descriptor was introduced in 1998 by Estrada, Torres, Rodr{\' i}guez and Gutman~\cite{etrg-abc-98}, who showed that it can be a valuable predictive tool in the study of the heat of formation in alkanes. Later the physico-chemical applicability of the ABC index was confirmed and extended in several studies~\cite{e-abceba-08, dt-cbfgaiabci-10, gg-nwabci-10, gtrm-abcica-12}. In addition, some comparative studies reported results that favour the ABC index over other mathematically similar molecular structure descriptors~\cite{fgd-ssdbti-13, gt-tqmsd-13}. As a consequence, the mathematical properties of the ABC index were studied with increasing interest over the last several years. More results about the mathematical and computational aspect of the ABC index can be found in~\cite{adgh-dctmabci-14, cg-eabcig-11, clg-subabcig-12, d-abcig-10, d-ectmabci-2013, d-sptmabci-2014, d-sptmabci-2-2015, ddf-sptmabci-3-2016, dgf-abci-11, fgiv-cstmabci-12, gfahsz-abcic-2013, gmng-abcitfnl-15, bcpt-nubfta-16, gs-tsaiotwnpv-16, hag-ktwmabci-14, lccglc-fcsftwmaibotdg-14, lcmzczj-otwmaiatwgnol-16, p-arubftai-14} and the references cited therein.

Several natural extensions and variants of the ABC index were introduced with a hope to obtain other successful molecular structure descriptors. Thus, in $2010$, Graovac and Ghorbani~\cite{gg-nwabci-10} proposed the following distance-based analog of the ABC index:
\begin{equation*}
\label{eq:GG}
    \GG(G) = \sum_{uv \in E(G)} \sqrt{\frac{n_u + n_v - 2}{n_u n_v}},
\end{equation*}
where $n_u$ is defined as the number of vertices of $G$ lying closer to $u$ than to $v$ and similarly $n_v$ as the number of vertices of $G$ lying closer to $v$ than to $u$, namely
\begin{equation}
\label{eq:nu_nv}
    n_u = | \{ w \in V(G) : d(w, u) < d(w, v) \} | \quad \text{and} \quad n_v = | \{ w \in V(G) : d(w, v) < d(w, u) \} |.
\end{equation}
Here $d(i, j)$ denotes the distance between vertices $i$ and $j$, i.e., the number of edges on the shortest path connecting $i$ and $j$.

This index was named after its authors as the \emph{Graovac-Ghorbani index}. Since it resembles the ABC index, it is also known as the \emph{second atom-bond connectivity index} and is often denoted by $\mathrm{ABC_2}$. However, as it was indicated in~\cite{f-abcvgga-2016}, this name and notation are inappropriate since the expression $\sqrt{(n_u + n_v - 2)/(n_u n_v)}$ is not related to atom connectivity (as the degree of a vertex) and bond connectivity (as the degree of a bond/edge). Therefore, we refer to this index as the Graovac-Ghorbani index and denote it by $\GG$. Some initial studies indicate that the Graovac-Ghorbani index could be an effective predictive tool in chemistry. For instance, it can be used to model both the boiling and the melting points of molecules~\cite{gw-sabcsms-2016}.

Due to its recentness, the Graovac-Ghorbani index was studied only to a limited extent. In~\cite{dxg-mugrnabci-13} maximal unicyclic graphs with respect to the $\GG$ index were determined. Some upper bounds on the $\GG$ index were presented in~\cite{rsg-sabci-13}. This work was extended in~\cite{rost-sohhag-14}, where lower and additional upper bounds for the $\GG$ index of graphs were given. The authors also determined the trees with minimal and maximal $\GG$ indices. Throughout the paper, we will need the result about trees with minimal $\GG$ indices. For later reference, we state it explicitly.
\begin{theorem}[Rostami, Sohrabi-Haghighat~\cite{rost-sohhag-14}] \label{thm:path-minGG_trees}
    The path $P_n$ is the $n$-vertex tree with the minimum Graovac-Ghorbani index.
\end{theorem}
Further upper and lower bounds on the $\GG$ index were presented in~\cite{d-ggig-16}. The $\GG$ indices of several special chemical molecular structures including the unilateral polyomino chain, the unilateral hexagonal chain, the V-phenylenic nanotubes and the nanotori were presented in~\cite{gw-sabcsms-2016}.
Recently, in~\cite{f-abcvgga-2016}, the structure of graphs with maximal $\GG$ was conjectured based on computational results. The relation between the $\GG$ index and the original ABC index was also investigated very recently in~\cite{f-abcvgga-2016, fgd-abmd-16, dmga-cbabcig-16}. In addition, in~\cite{f-abcvgga-2016}, some physico-chemical properties of the $\GG$ index were compared with the physico-chemical properties of a few other well-known distance-based molecular descriptors, and the obtained results suggest that it might be worth investigating the $\GG$ index further.

To simplify the derivation of the results presented in the paper, we find it useful to introduce the notion of a \emph{normalized Graovac-Ghorbani index}. We denote it by $\NGG$, and define it as
\begin{equation*}
    \NGG(G) = \sum_{uv \in E(G)} \frac{1}{\sqrt{n_u n_v}}.
\end{equation*}
This measure is of purely technical importance to us. The Graovac-Ghorbani index $\GG$ can be expressed very conveniently using $\NGG$ when dealing with bipartite graphs.
\begin{proposition}
\label{prop:NGG_bprt}
    Let $G$ be a bipartite graph on $n$ vertices. Then
        \begin{equation*}
            \GG(G) = \NGG(G) \sqrt{n-2}.
        \end{equation*}
\end{proposition}
\begin{proof}
    Following the definition of $n_u$ and $n_v$ in~\eqref{eq:nu_nv}, one can conclude that in the case of a bipartite graph, $n_u + n_v = n$ holds for any edge $uv$ of $G$. Indeed, assume that this does not hold, i.e., assume that $n_u + n_v < n$ for some edge $uv \in E(G)$. Then there exists a vertex $w$ such that $d(u, w) = d(v, w)$ and therefore $G$ contains an odd cycle, which in turn is a contradiction to the fact that $G$ is a bipartite graph.

    Hence, the desired equality holds:
    \begin{equation*}
        \GG(G) = \sum_{uv \in E(G)} \sqrt{\frac{n-2}{n_u n_v}} = \sqrt{n-2} \cdot \NGG(G).
    \end{equation*}
\end{proof}
Notice that Proposition~\ref{prop:NGG_bprt} implies that in any subclass of bipartite graphs of fixed order $n$, the extremal graphs for the $\NGG$ index and the extremal graphs for the $\GG$ index are the same.

In the rest of the paper, we first consider the $\GG$ index of very long paths in Section~\ref{sec:paths}. In Section~\ref{sec:bipartite-graphs}, we present all bipartite graphs with minimal and maximal $\GG$ indices. We conclude and state some open problems in Section~\ref{sec:conclusion}.

\section{Graovac-Ghorbani index of paths} \label{sec:paths}

It is well-known that, in general, the graph with $n$ vertices that minimizes the Graovac-Ghorbani index is the complete graph $K_n$~\cite{f-abcvgga-2016}. From a chemical point of view, it is also interesting to confine oneself to studying trees. By Theorem~\ref{thm:path-minGG_trees}, a path is a tree with minimal $\GG$ index. In~\cite{rost-sohhag-14}, it was shown that a star graph is a tree with maximal $\GG$ index. Nevertheless, graphs with a lower maximum degree tend to be more relevant to applications in chemistry as the maximum degree of a \emph{molecular graph}, that is, a topological representation of a (typically organic) molecule, never exceeds four.

For this reason, we spend some time examining paths. In the limit, as the length of the path tends to infinity, a nice and simple result follows.

\begin{proposition}
    For the path $P_n$ on $n$ vertices, the ensuing relation holds:
        \begin{equation*}
            \lim_{n \rightarrow \infty} \NGG(P_n) = \pi.
        \end{equation*}
\end{proposition}

\begin{proof}
    Let us sketch why the statement is true. We have that
        \begin{equation*}
            \NGG(P_n) = \sum_{i = 1}^{n-1} \frac{1}{\sqrt{i (n-i)}} = \frac{1}{n} \displaystyle\sum_{i = 1}^{n-1} \frac{1}{\sqrt{\frac{i}{n}(1-\frac{i}{n})}}.
        \end{equation*}
    Observe that the latter form of the sum can serve as an approximation of the area bounded by the curves $f(x) = (x(1-x))^{-1/2}$, $x = 0$, $x = 1$ and the $x$-axis. Evidently, the greater the $n$, the more precise the approximation is. The result now follows from the existence of the improper integral $\int_0^1 f(x) dx$, which is known to be equal to $\pi$.
\end{proof}

From the result above, one can easily estimate the value of $\GG(P_n)$ for large $n$.

\begin{corollary}
    Let $P_n$ be a path on $n$ vertices. It holds that
    \begin{equation*}
        \GG(P_n) \sim \pi \sqrt{n-2}.
    \end{equation*}
\end{corollary}

\section{Extremals among bipartite graphs} \label{sec:bipartite-graphs}

We will now move on to the Graovac-Ghorbani index of more general bipartite graphs and provide some results concerning extremal problems. More specifically, we will identify all bipartite graphs with maximal and minimal $\GG$ indices.

\begin{theorem}
\label{thm:bipartite_max}
    Amongst all bipartite graphs on $n$ vertices, the maximum (normalized) Graovac-Ghorbani index is uniquely attained by the complete bipartite graph $K_{\floor{n/2}, \ceil{n/2}}$.
\end{theorem}

\begin{proof}
    Choose an arbitrary connected bipartite graph $G$ whose vertex set is divided into two disjoint sets $U$ and $V$ such that $|U|=n_1$ and $|V|=n_2$ with $n_1 + n_2 = n$. Moreover, without loss of generality, assume that for every edge $uv \in E(G)$, inclusions $u \in U$ and $v \in V$ hold. Let us investigate the square of the normalized Graovac-Ghorbani index and find an upper bound for it. Note that, by the Cauchy-Schwarz inequality,
    \begin{equation*}
        \NGG^2(G) = \left( \sum_{uv \in E(G)} \frac{1}{\sqrt{n_u}} \cdot \frac{1}{\sqrt{n_v}} \right)^2 \leq \left( \sum_{uv \in E(G)} \frac{1}{n_u} \right) \left( \sum_{uv \in E(G)} \frac{1}{n_v} \right),
    \end{equation*}
    where we assume $u \in U$ and $v \in V$.

    Furthermore, observe that for each edge $uv \in E(G)$, inequalities $n_u \geq d_u$ and $n_v \geq d_v$ have to be fulfilled as the set of vertices lying closer to either of the vertices (in comparison to the other vertex) has to contain at least the vertex itself and all its immediate neighbours except for $v$ (resp. $u$). The latter is an evident consequence of the bipartiteness of $G$.

    Hence,
    \begin{equation*}
        \left( \sum_{uv \in E(G)} 1/n_u \right) \left( \sum_{uv \in E(G)} 1/n_v \right) \leq \left( \sum_{uv \in E(G)} 1/d_u \right) \left( \sum_{uv \in E(G)} 1/d_v \right) = n_1 \cdot n_2.
    \end{equation*}
    The last equality follows from the fact that each vertex $u \in U$ contributes exactly one to the first sum and similarly each vertex $v \in V$ contributes one to the second sum.

    Under the constraint that $n_1+n_2=n$, the product $n_1 \cdot n_2$ attains its maximum when $n_1=\floor{n/2}$ and $n_2=\ceil{n/2}$ (or vice versa). Notice that $n_1 = n_2 = n/2=\floor{n/2}= \ceil{n/2}$ when $n$ is even. Thus, the maximum of $\NGG^2(G)$ and, consequently, the maximum of $\NGG(G)$ can only be achieved when $G$ is a complete bipartite graph $K_{\floor{n/2}, \ceil{n/2}}$. By Proposition~\ref{prop:NGG_bprt}, the same result follows for the $\GG$ index.
\end{proof}

As an immediate consequence of Theorem~\ref{thm:bipartite_max}, we have the following corollary.

\begin{corollary}
    The maximum values of the normalized Graovac-Ghorbani index and the Graovac-Ghorbani index over the set of all bipartite graphs on $n$ vertices are
        \begin{equation*}
           \NGG(G) = \sqrt{\floor{n/2} \ceil{n/2}} \quad \text{and} \quad
            \GG(G) = \sqrt{(n-2) \floor{n/2} \ceil{n/2}},
        \end{equation*}
    respectively. Both values are achieved if and only if $G$ is the complete bipartite graph $K_{\floor{n/2}, \ceil{n/2}}$.
\end{corollary}

A natural problem that arises from the abovementioned theorem is the characterization of bipartite graphs that minimize the $\GG$ index.

For the purpose of the next theorem, which provides an answer to the problem posed, we introduce the notion of a \emph{cycle with a hook}, that is, a graph with an odd number of vertices $n$ comprised of two even cycles $C_{n-1}$ and $C_4$ which share three common vertices and two common edges, and write $C''_n$. Moreover, we denote by $C'_n$ a \emph{cycle with a pendant edge}, that is, an unicyclic graph with an odd number of vertices $n$ comprised of a cycle $C_{n-1}$ and a pendant vertex. See Figure~\ref{fig-Cn-plus1plus2} for an illustration of $C'_n$ and $C''_n$.

The {\em cyclomatic number} (or the {\em circuit rank}) of an undirected graph $G$ is the minimum number of edges whose removal from $G$ breaks all its cycles, making it into a tree or a forest. The cyclomatic number $r$ can be expressed as $r = |E(G)| - |V(G)| + |C(G)|$, where $|C(G)|$ is the number of connected components of $G$. Since we are interested in connected graphs, we will assume that $|C(G)| = 1$.
\begin{figure}[H]
    \begin{center}
        \includegraphics[scale=0.65]{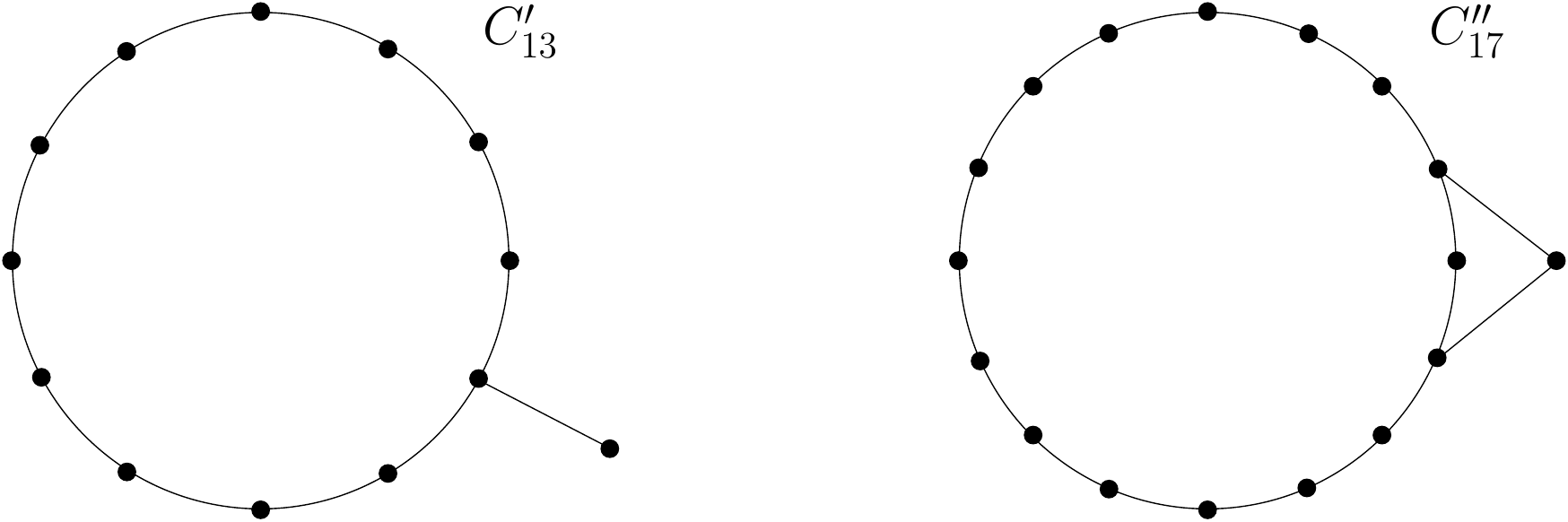}
        \caption{Examples of graphs $C'_n$ and $C''_n$.}
        \label{fig-Cn-plus1plus2}
    \end{center}
\end{figure}
\begin{theorem}
\label{thm:min-bipartite}
    Amongst all bipartite graphs on $n \geq 8$ vertices, the minimum Graovac-Ghorbani index is attained by the cycle $C_n$ for even $n$, by $C'_n$ for odd $n \leq 15$ and by $C''_n$  for odd $n \geq 17$.
    For $n < 8$, the graph that minimizes the Graovac-Ghorbani index is the path $P_n$ on $n$ vertices.
    Furthermore, these are the unique graphs with the respective properties.
\end{theorem}

\begin{proof}
    Let $G$ denote a bipartite graph on $n$ vertices that minimizes the $\GG$ index.
    The paths $P_2$ and $P_3$ are the only connected bipartite graphs on $2$ and $3$ vertices respectively.
    For the rest of the proof we will therefore assume that $n \geq 4$.
    Again, it suffices to consider the normalized index $\NGG$ only as we already know that $n_u + n_v = n$ holds for every edge $uv$ in a bipartite graph.

    The calculations of the normalized indices of the cycle $C_n$, the cycle with a pendant edge $C'_n$, and the cycle with a hook $C''_n$ are straightforward, hence we omit them. We will show that the minimum value of the normalized index in the class of all bipartite graphs indeed is attained by these three graphs for $n \geq 8$ and by $P_n$ for $n < 8$. According to the parity of $n$, we split the proof in the ensuing two cases.

    \textbf{Case 1:} \emph{$n$ is even}, $n = 2k, k \in \mathbb{N}$, $k > 1$. Note that $1/\sqrt{n_u n_v} \geq 1/k$ for every edge $uv$ in $G$.

    Suppose first that $G$ is not a tree.
    Then $G$ contains at least $n$ edges. Consequently, $\NGG(G) \geq n/k = 2 = \NGG(C_n)$, where equality holds whenever $|E(G)|= n$ and $1/\sqrt{n_u n_v} = 1/k$ for all $uv \in E(G)$. The first condition amounts to $G$ being unicyclic whilst the second one forces $G$ to be precisely the cycle $C_n$. Otherwise, $G$ would contain a leaf $w \in V(G)$. Evidently, $1/\sqrt{n_u n_w} = 1/\sqrt{2k-1} > 1/k$ holds for the vertex $u$ that shares a common edge with a leaf $w$. However, this cannot be as $G$ is supposed to be the graph that minimizes $\NGG(G)$. Hence, a contradiction.

    On the other hand, if $G$ is a tree, it has to be a path $P_n$ on $n$ vertices by Theorem~\ref{thm:path-minGG_trees}.
    Thus, our proof boils down to determining for which $n$ the inequality $\NGG(P_n) > \NGG(C_n)$, that is,
    \begin{equation}
    \label{eq:NGGPn>NGGCn}
        1/\sqrt{(n-1)} \, + \, 1/\sqrt{2(n-2)} \,  + \, \cdots \, + \, 1/\sqrt{(n-1)} \, > \, 2/n \, + \, 2/n \, +\,  \cdots\,  + \, 2/n \, = \, 2,
    \end{equation}
    holds. To find out, we use the inequality
    \begin{equation*}
        2/\sqrt{(n-1)} + 2/\sqrt{2(n-2)} > 5 \cdot 2/n,
    \end{equation*}
    which is valid for all $n \geq 7$. By subtracting it from the inequality~\eqref{eq:NGGPn>NGGCn}, we are left with an inequality with $n-5$ terms on each side. The left-hand side of the resulting inequality clearly exceeds its right-hand side as $1/\sqrt{i(n-i)} \geq 2/n$ for all $0 < i < n$ with equality if and only if $i = n/2$. Therefore, $\NGG(P_n) > \NGG(C_n)$ for all (even) $n \geq 7$.

    One can verify that for $4 \leq n < 7$, inequality $\NGG(P_n) < \NGG(C_n)$ holds.
    For a comparison between the normalized Graovac-Ghorbani index of the cycle $C_n$, which is equal to $2$ for even $n$, and the normalized Graovac-Ghorbani index of the path $P_n$ (for even $n$), we refer the reader to Table~\ref{table:NGG-Pn}.
    \begin{table}[H]
    \centering
        \begin{tabular}{cccccccc}
            \toprule
            $n$ & 4 & 6 & 8 & 10 & 12 & 14 & 16 \\
            $\NGG(P_n)$ & 1.6547 & 1.9349 & 2.0997 & 2.2114 & 2.2934 & 2.3570 & 2.4081 \\
            \midrule
            & 18 & 20 & 22 & 24 & 26 & 28 & 30 \\
            & 2.4504 & 2.4862 & 2.5169 & 2.5436 & 2.5672 & 2.5882 & 2.6071 \\
            \bottomrule
        \end{tabular}
        \caption{Numerical values of $\NGG(P_n)$ for even $n$, $4 \leq n \leq 30$.}
        \label{table:NGG-Pn}
    \end{table}

    Thus, we have shown that among bipartite graphs on an even number of vertices $n$, the minimum $\GG$ index is attained by the path $P_n$ for $n < 8$ and the cycle $C_n$ for $n \geq 8$.

    \textbf{Case 2:} \emph{$n$ is odd}, $n = 2k+1$, $k \in \mathbb{N}$, $k > 1$. We will divide the analysis into two subcases:

    \textbf{Subcase 2.1:} \emph{$G$ is not a tree.} Hence, it contains a cycle. First of all, we will devote our attention to $G$ being unicyclic. Clearly, its set of leaves has to be non-empty as the cycle $C_n$ on $n = 2k+1$ vertices fails to be bipartite.

    Let $w$ be a leaf of $G$. As already discussed in the previous case, $1/\sqrt{n_u n_w} = 1/\sqrt{n-1} = 1/\sqrt{2k}$ for every edge $uw \in E(G)$ with a leaf $w$ as one of its endpoints. Moreover, $1/\sqrt{n_u n_v} \geq 1/\sqrt{k(k+1)}$ for every (other) edge $uv \in E(G)$. Thus,
    \begin{equation*}
        \NGG(G) \geq 1/\sqrt{2k} + 2k/\sqrt{k(k+1)} = \NGG(C'_n).
    \end{equation*}
    As a result, the cycle with a pendant edge $C'_n$ has the lowest normalized $\GG$ index in the class of unicyclic bipartite graphs. Indeed, all other graphs in the same class contain a pendant path of length at least two (and hence an edge that contributes $1/\sqrt{2(n-2)}$, which in this case is strictly greater than $1/\sqrt{k(k+1)}$) or at least two leaves, and thus have a higher normalized $\GG$ index. Notice that the order of $C'_n$ is at least $5$.

    Let us now focus on graphs which may contain more than one cycle. If $G$ is not unicyclic, $|E(G)| \geq 2k+2$ and so $\NGG(G) \geq (2k+2)/\sqrt{k(k+1)} = \NGG(C''_n)$. We need to verify that $C''_n$ is the only graph that attains this lower bound. Observe that any alternative graph $G$ which fulfills this condition has the following properties: it has precisely $2k+2$ edges and none of its vertices is a leaf as $1/\sqrt{2k} > 1/\sqrt{k(k+1)}$ for $k \geq 2$ (the case $k = 1$ is not of our interest as it corresponds to an odd cycle, which is not a bipartite graph). Thus, $d(v) \geq 2$ for each $v \in V(G)$.
    Since $G$ has $2k+1$ vertices, $2k+2$ edges and one connected component, its cyclomatic number is $2$. There are three possible classes of graphs with cyclomatic number $2$, one connected component and no pendant paths, and they are all illustrated in Figure~\ref{fig-cyclomatic-nr2}. We denote them by Configuration $(a)$, $(b)$ and $(c)$, and consider them separately.

    \begin{figure}[!ht]
    \begin{center}
        \includegraphics[scale=1.1]{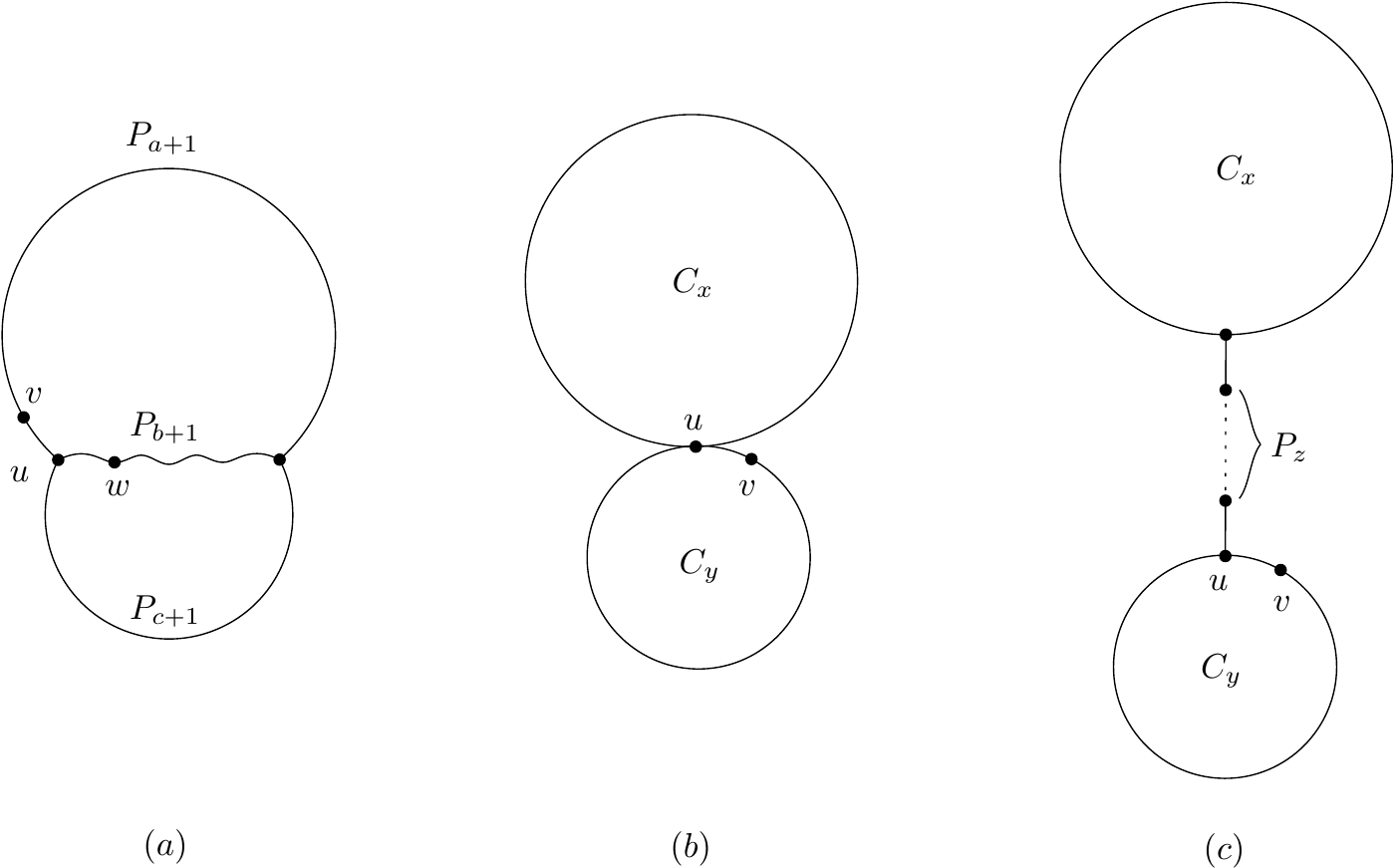}
        \caption{Three possible configurations of graphs with cyclomatic number $2$, one connected component and no pendant paths. }
        \label{fig-cyclomatic-nr2}
    \end{center}
    \end{figure}

    \begin{itemize}

    \item {\it Configuration $(a)$.}

        This configuration is also known as a $\Theta$-graph, that is, a graph that consists of three paths $P_{a+1}$, $P_{b+1}$ and $P_{c+1}$ with shared endpoints. The lengths of these paths are $a, b$ and $c$ respectively. Taking into account that $a + b + c = 2k+2$ and that $a \equiv b \equiv c \pmod{2}$ (which is a consequence of bipartiteness of $G$), we deduce that the lengths $a$, $b$ and $c$ should be even. We assume that $a \geq b \geq c$.

        If two of these lengths equal two, then $G$ is isomorphic to $C''_n$.
        Consider now the case where two of the lengths $a, b, c$ are larger than two, i.e., two of them are at least four. Together with $a \geq b \geq c$, we can conclude that $a, b \geq 4$.
        Here we distinguish three possibilities.

        Suppose first that $b >c$. In this case, observe the edge $uw$ of the graph $G$ depicted in Figure~\ref{fig-cyclomatic-nr2}($a$).
        We have that $n_u = (b+c)/2+a-1$ and $n_w=(b+c)/2$. Since $a \geq 4$, it follows that $n_u > k+1$ and $n_w < k$.
        Consequently, $1/\sqrt{n_u n_w} > 1/\sqrt{k(k+1)}$ and therefore $\NGG(G) > \NGG(C''_n)$.

        Now suppose that $a >b$. Here we assume that $b=c$, since otherwise we are in the preceding case.
        Now observe the edge $uv$ of the graph $G$ in Figure~\ref{fig-cyclomatic-nr2}($a$). Evidently, equalities $n_u = (a+b)/2+c-1$ and $n_v = (a+b)/2$ hold. Recall that $b \geq 4$ and thus $c \geq 4$. It follows that $n_u > k+1$ and $n_v < k$, and subsequently $1/\sqrt{n_u n_v} > 1/\sqrt{k(k+1)}$. Hence, $\NGG(G) > \NGG(C''_n)$.

        Finally, we may assume that $a=b=c$. Similarly as in the case $b > c$, we consider the edge $uw$ of $G$. Obviously, $n_u = 2a-1$ and $n_w=a$. Yet again, this edge contributes more than $1/\sqrt{k(k+1)}$ to $\NGG(G)$ and thus $\NGG(G) > \NGG(C''_n)$.

    \item {\it Configurations $(b)$ and $(c)$.}

        The analyses of these two configurations are similar, so we study them together.
        Configuration $(b)$ consists of two cycles $C_x$ and $C_y$ of even order that share a common vertex while Configuration $(c)$ is a graph comprised of two cycles $C_x$ and $C_y$ of even order and a path $P_z$ of odd order.
        Indeed, observe that both cycles $C_x$ and $C_y$ must be of even order, otherwise the graph is not bipartite.
        Consequently, because $G$ is of odd order, the path $P_z$ must be of odd order too.
        We may assume that in both configurations the order of $C_x$ is at least as large as the order of $C_y$, i.e., $x \geq y$.
        Let us now consider the edges denoted by $uv$.
        For the edge $uv$ of Configuration $(b)$, it holds that $n_u = x-1+y/2$ and $n_v=y/2$, and for the edge $uv$ of Configuration $(c)$, we have $n_u = x+z+y/2$ and $n_v = y/2$.
        Since $x \geq 4$, it follows that in both cases inequalities $n_u > k+1$ and $n_w < k$ hold, and therefore $1/\sqrt{n_u n_v} > 1/\sqrt{k(k+1)}$ and $\NGG(G) > \NGG(C''_n)$.
    \end{itemize}

    Accordingly, in order to finish this subcase, we need to compare the values $\NGG(C'_n)$ and $\NGG(C''_n)$. A simple calculation reveals that $\NGG(C'_n) > \NGG(C''_n)$ if and only if $k \geq 8$, that is, $n \geq 17$.
    However, bear in mind that we still have to compare the acquired normalized indices to those of the trees so as to conclude the proof.

    \textbf{Subcase 2.2:} \emph{$G$ is a tree.} As aforementioned, it follows by Theorem~\ref{thm:path-minGG_trees} that $G$ is a path $P_n$.
    Hence, we proceed in a similar vein as in the first case.
    For this reason, we begin by rewriting the inequality $\NGG(P_n) > \NGG(C'_n)$ in terms of
    \begin{equation*}
        \sum_{i=1}^{2k} 1/\sqrt{i(2k+1-i)} > 1/\sqrt{2k} + 2k/\sqrt{k(k+1)}.
    \end{equation*}
    The validity of this inequality for $k \geq 4$ is an immediate consequence of the relation
    \begin{equation*}
        1/\sqrt{2k} + 2/\sqrt{2(2k-1)} + 1/\sqrt{3(2k-2)} > 5/\sqrt{k(k+1)},
    \end{equation*}
    which holds for all $k \geq 4$ (or equivalently, $n \geq 9$), and the chain of inequalities
    \begin{equation*}
        1/\sqrt{2k} > 1/\sqrt{2(2k-1)} > \cdots > 1/\sqrt{k(k+1)}.
    \end{equation*}

    Direct calculation shows that $\NGG(P_7) \approx 2.0263 < 2.1403 \approx \NGG(C'_7)$ and $\NGG(P_5) \approx 1.8165 < 2.1330 \approx \NGG(C'_5)$.

    To recap: for odd $n$, $\NGG(P_n) < \NGG(C'_n) < \NGG(C''_n)$ for $n \leq 7$, $\NGG(C'_n) < \NGG(P_n)$ and $\NGG(C'_n) < \NGG(C''_n)$ for $9 \leq n \leq 15$, and $\NGG(C''_n) < \NGG(C'_n) < \NGG(P_n)$ for $n \geq 17$.
\end{proof}

As a direct consequence of the proof of Theorem~\ref{thm:min-bipartite}, we obtain the minimum values of the (normalized) Graovac-Ghorbani index.

\begin{corollary}
    The minimum values of the normalized Graovac-Ghorbani index over the set of all bipartite graphs on $n \geq 8$ vertices are
    \begin{equation*}
        \NGG(G) =  \GG(G)/\sqrt{n-2}=
            \left\{
                \begin{array}{ll}
                    2 & \quad \text{for} \ n \geq 8 \ \text{even}, \\
                    (n-1)^{-1/2} + (n-1) \cdot N^{-1} & \quad \text{for} \ 9 \leq n \leq 15 \ \text{odd}, \\
                    (n+1) \cdot N^{-1} & \quad \text{for} \ n \geq 17 \ \text{odd},
                \end{array}
            \right.
    \end{equation*}
    where $N = \displaystyle\sqrt{\floor{n/2} \ceil{n/2}}$.
\end{corollary}

\section{Concluding comments} \label{sec:conclusion}

The Graovac-Ghorbani index is a distance-based analog of the atom-bond connectivity index, one of the most meaningful degree-based molecular structure descriptors.
In this work, we have characterized extremal bipartite graphs with respect to the Graovac-Ghorbani index.
There remain several open problems regarding the extremal graphs within other classes of graphs.
Here we would like to draw attention to extremal graphs and trees with a given maximum degree $\Delta \ll n$.
Recall that these graphs are of significant importance in chemical and pharmacological applications.
In the sequel, we present three conjectures.

For the first one, we need the notion of an (almost) $k$-regular graph. A graph is said to be {\em regular} if all vertices have the same degree.
A regular graph is $k$-regular if every vertex has degree $k$. We say that a graph is {\em almost $k$-regular} if all vertices have degree $k$ except for one which has degree $k-1$.

\begin{conjecture} \label{conj1}
    Let $G$ be a graph with maximal $\GG$ index amongst all graphs on $n \gg \Delta$ vertices.
    Then $G$ is an (almost) $\Delta$-regular graph.
\end{conjecture}

In contrast, in the second conjecture, we consider graphs with a given maximum degree $\Delta \ll n$ that minimize the Graovac-Ghorbani index.

\begin{conjecture} \label{conj2}
    Let $G$ be a graph with minimal $\GG$ index amongst all graphs on $n \gg \Delta$ vertices. Then $G$ is the cycle $C_n$.
\end{conjecture}

In order to formulate the last conjecture, we need the following definitions.
The {\em level or depth} of a vertex in a rooted tree is the length of the path from the vertex to the root of the tree.
In a {\em breadth-first traversal} of a rooted tree, the vertices are visited level by level from the root to the bottom, each level being traversed from left to right.
An {\em almost dendrimer} $T_{n, d}$ is a rooted tree with $n$ vertices in which every non-pendant vertex, except perhaps one, has degree $d$ and in which inequality $d(u) \geq d(v)$ holds for every vertex $u$ that occurs before vertex $v$ in the breadth-first traversal (cf. Figure~\ref{fig-dendrimer} for an illustration).

\begin{figure}[H]
    \begin{center}
        \includegraphics[scale=0.65]{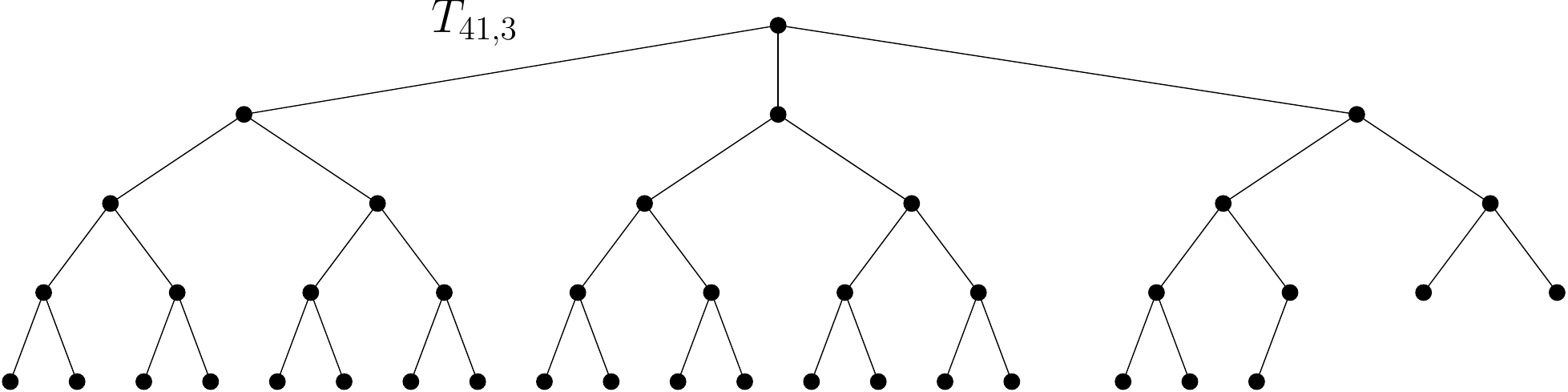}
        \caption{An almost dendrimer $T_{41,3}$.}
        \label{fig-dendrimer}
    \end{center}
\end{figure}

\begin{conjecture} \label{conj3}
    Let $G$ be a tree with maximal $\GG$ index amongst all trees on $n$ vertices with maximum degree $\Delta \leq n-1$. Then $G$ is
    an almost dendrimer $T_{n, \Delta}$.
\end{conjecture}

\bigskip

\noindent
{\bf{Acknowledgements.} }We are grateful to Matjaž Konvalinka for helpful comments and suggestions.

The research was supported by the ARRS Program P1-0383.

%
%
 %

\end{document}